\documentclass[12pt,twoside,reqno]{amsart}
\usepackage{amsmath}
\usepackage{amsfonts}
\usepackage{amssymb}
\usepackage{color}
\usepackage{mathrsfs}
\usepackage{cite}
\usepackage{cleveref}
\usepackage{geometry}
\usepackage{marginnote}%\marginpar{}, \marginnote{}
\usepackage{todonotes}%\todo{}
\usepackage{color}

\allowdisplaybreaks
\newcommand{\rd}{\mathrm{d}}
\newcommand{\dom}{\mathrm{dom}}

\newtheorem{theorem}{Theorem}[section]

\newtheorem{lemma}[theorem]{Lemma}
\newtheorem{proposition}[theorem]{Proposition}

\newtheorem{remark}[theorem]{Remark}
\numberwithin{equation}{section}
\begin{document}
\title{Well-Posedness of the Coagulation-Fragmentation Equation with Size Diffusion} 
%
%%%%%%%%%%%%%%%%
\author{Philippe Lauren\c{c}ot}
\address{Institut de Math\'ematiques de Toulouse, UMR~5219, Universit\'e de Toulouse, CNRS \\ F--31062 Toulouse Cedex 9, France}
\email{laurenco@math.univ-toulouse.fr}
\author{Christoph Walker}
\address{Leibniz Universit\"at Hannover\\ Institut f\"ur Angewandte Mathematik \\ Welfengarten 1 \\ D--30167 Hannover\\ Germany}
\email{walker@ifam.uni-hannover.de}
%%%%%%%%%%%%%%%%
%	
\date{\today}
%	
%%%%%%%%%%%%%%%%
%%%%%%%%%%%%%%%%
\begin{abstract}
Local and global well-posedness of the coagulation-fragmentation equation with size diffusion are investigated. Owing to the semilinear structure of the equation, a semigroup approach is used, building upon generation results previously derived for the linear fragmentation-diffusion operator in suitable weighted $L_1$-spaces.
\end{abstract}
%%%%%%%%%%%%%%%%
%%%%%%%%%%%%%%%%
%	
\maketitle
%	
%%%%%%%%%%%%%%%%
%     HEADLINES
%%%%%%%%%%%%%%%%
%
\pagestyle{myheadings}
\markboth{\sc{Ph.~Lauren\c{c}ot \& Ch.~Walker}}{\sc{Coagulation-Fragmentation Equation with Size Diffusion}}
%	
%
%%%%%%%%%%%%%%%%
%%%%%%%%%%%%%%%%
\section{Introduction}\label{S.1}
%%%%%%%%%%%%%%%%
%%%%%%%%%%%%%%%%

The coagulation-fragmentation equation with size diffusion 
\begin{subequations}\label{FD.0}
	\begin{align}
		\partial_t \phi(t,x) & = D \partial_x^2 \phi(t,x) +  \mathcal{F}(\phi(t,\cdot))(x) + \mathcal{K}(\phi(t,\cdot))(x) \,, \qquad (t,x)\in (0,\infty)^2\,, \label{FD.1} \\
		\phi(t,0) & = 0\,, \qquad t>0\,, \label{FD.2} \\
		\phi(0,x) & = f(x)\,, \qquad x\in (0,\infty)\,, \label{FD.3} 
	\end{align}
\end{subequations}
where
\begin{equation*}
	\mathcal{F}(\phi)(x):=- a(x) \phi(x) + \int_x^\infty a(y) b(x,y) \phi(y)\ \mathrm{d}y
\end{equation*}
and $\mathcal{K}(\phi):=\mathcal{K}(\phi,\phi)$ with
\begin{equation*}
\mathcal{K}(\phi,\psi)(x):=\frac{1}{2}\int_0^x k(y,x-y) \phi(y)\psi(x-y)\,\mathrm{d}y-\phi(x)\int_0^\infty k(x,y)\psi(y)\,\mathrm{d}y\,,
\end{equation*}
describes the dynamics of the size distribution function $\phi=\phi(t,x)\ge 0$ of particles of size $x\in (0,\infty)$ at time $t>0$. Particles modify their sizes according to three different mechanisms: random fluctuations, here accounted for by size diffusion at a constant diffusion rate $D>0$ (hereafter normalized to $D=1$), spontaneous fragmentation with overall fragmentation rate $a\ge 0$ and daughter distribution function $b\ge 0$, and binary coalescence with coagulation kernel $k\ge 0$. Nucleation is not taken into account in this model, an assumption which leads to the homogeneous Dirichlet boundary condition \eqref{FD.2} at $x=0$. Let us recall that the coagulation-fragmentation equation without size diffusion, corresponding to setting $D=0$ in \eqref{FD.1}, arises in several fields of physics (grain growth, aerosol and raindrops formation, polymer and colloidal chemistry) and biology (hematology, animal grouping) and has been studied extensively in the mathematical literature since the pioneering works by Melzak \cite{Mel57} and McLeod \cite{McL64}, see the books and review articles \cite{Ald1999, Ban2006, BLL2020a, BLL2020b, Ber2006, Dub94, LaMi2004, Ram2000, Wat2006} and the references therein. The fragmentation equation with size diffusion ($k=0$) has been introduced more recently to describe the growth of microtubules \cite{FHL1994}, plankton cells \cite{BWWvB2004}, and ice crystals \cite{FJMOS2003, MFJLODS2004}. That merging could also play a role in the dynamics of the latter is suggested and investigated in \cite{FJMO2006, OFJM2005} and the aim of this paper is to provide some mathematical insight into this model. Specifically, we shall study the well-posedness of \eqref{FD.0} in a suitable functional setting for a class of rate coefficients $a$, $b$, and $k$ which will be described below.

To this aim, let us first recall that, owing to the semilinear structure of \eqref{FD.1}, a classical approach to well-posedness relies on semigroup theory and involves two steps: one first shows that the linear part of  \Cref{FD.0} (including size diffusion and fragmentation, along with the boundary condition \eqref{FD.2}) generates a semigroup in a suitable function space. One subsequently establishes the well-posedness of \eqref{FD.0} by a fixed point procedure applied to the associated Duhamel formula taking into account Lipschitz properties of the coagulation term. This approach has proven successful for the coagulation-fragmentation equation ($D=0$) since the seminal work of Aizenman \& Bak \cite{AiBa1979}, see \cite{Ban2020, BaLa2012, BLL2013, BLL2020b, MLM1997} and the references therein. It also leads to the well-posedness of coagulation-fragmentation equations with growth/decay (where the diffusion $D\partial_x^2 f$ is replaced by $\mp \partial_x(gf)$ for some growth rate $g=g(x)\ge 0$), see \cite{BaLa2020}.

Coming back to the coagulation-fragmentation equation with size diffusion \eqref{FD.0}, we have performed a rather complete study of the generation properties of its linear part in \cite{LaWa2021} and shown that it generates a positive analytic semigroup in $L_1((0,\infty),(x+x^m)\mathrm{d}x)$ for any $m>1$, thereby setting the stage for the study of the well-posedness of \eqref{FD.0} which is our main concern herein. Beforehand, let us set up some notation and make precise the assumptions on the  rate coefficients $a$, $b$, and $k$ that we shall use throughout the paper.

%%%%%%%%%%%%%%%%
%%%%%%%%%%%%%%%%
\subsection*{Notation and Assumptions}
%%%%%%%%%%%%%%%%
%%%%%%%%%%%%%%%%

We suppose that the overall fragmentation rate $a$ satisfies
\begin{equation}
a\in L_{\infty,loc}([0,\infty))\,, \qquad a\ge 0 \;\text{ a.e. in }\; (0,\infty)\,. \label{A.0}
\end{equation}
The daughter distribution function $b$ is a non-negative measurable function on $(0,\infty)^2$ satisfying
\begin{equation}
	\int_0^y x b(x,y)\ \mathrm{d}x = y\,, \qquad y\in (0,\infty)\,, \label{B.0}
\end{equation}
and there is $\delta_2\in (0,1)$ such that
\begin{equation}
	(1-\delta_2) y^2 \ge \int_0^y x^2 b(x,y)\ \mathrm{d}x\,, \qquad y\in (0,\infty)\,. \label{B.10}
\end{equation}
Recall that \eqref{B.0} guarantees that there is no loss of matter during fragmentation, while \eqref{B.10} implies that the distribution of the sizes of the fragments resulting from the breakup of a particle of size $y$ is not too concentrated around $y$. As for the coagulation kernel $k$, we assume that there are $0\le \theta_0 < \theta \le 1$, $m>1$, and $k_*>0$ such that
\begin{subequations}\label{K1L}
\begin{equation}\label{K1}
	0\le k(x,y)=k(y,x)\le k_*\frac{\ell(x)\ell(y)}{x+y+(x+y)^m}\,,\qquad (x,y)\in (0,\infty)^2\,,
\end{equation}
where
\begin{equation}
	\ell(x):=\left\{
	\begin{array}{ll} 
		x^{1-2\theta_0}\,, & x\in (0,1)\,,\\
		(1+a(x))^\theta x^m\,, & x>1\,. 
	\end{array}\right. \label{L}
\end{equation}
\end{subequations}
As usual for coagulation-fragmentation equations, the analysis is performed in weighted $L_1$-spaces, which we introduce next. For $r\in\mathbb{R}$, we set
\begin{equation*}
	X_r := L_1\big( (0,\infty),x^r\mathrm{d}x \big)\,, \qquad \|f\|_{X_r} := \int_0^\infty x^r |f(x)|\ \mathrm{d}x\,, \quad f\in X_r\,,
\end{equation*} 
and define the moment $M_r(f)$ of order $r$ of $f$ by
\begin{equation*}
M_r(f) := \int_0^\infty x^r\ f(x)\ \mathrm{d}x\,,
\end{equation*}
so that $\|f\|_{X_r}=M_r(|f|)$. We then define the following spaces, which are at the heart of our analysis, namely
\begin{equation*}
	E_0:=X_1\cap X_m=L_1\big( (0,\infty),(x+x^m)\mathrm{d}x \big)
\end{equation*}
equipped with the norm $\|\cdot\|_{E_0} := \|\cdot\|_{X_1} + \|\cdot\|_{X_m}$ and 
\begin{equation}
	Y := L_1\big((0,\infty),\ell(x)\mathrm{d}x\big) = X_{1-2\theta_0} \cap L_1\big((0,\infty),(1+a(x))^\theta (x+x^m)\mathrm{d}x\big)\,, \label{SY}
\end{equation} 
with norm
\begin{equation*}
	\|f\|_Y := \int_0^\infty \ell(x) |f(x)|\ \mathrm{d}x\,, \qquad f\in Y\,,
\end{equation*}
recalling that the parameters $(\theta_0,\theta,m)$ and the weight $\ell$ are defined in \eqref{K1} and \eqref{L}, respectively.

\medskip

We next introduce the linear operator
\begin{equation*}
	\mathbb{A} f:=\partial_x^2 f +\mathcal{F}(f)\,,\qquad f\in \dom(\mathbb{A})\,,
\end{equation*}
where
\begin{equation*}
	\dom(\mathbb{A}) := \{ f\in E_0\ :\ \partial_x^2 f\in E_0\,, \ af \in E_0\,, \ f(0)=0\}\,,
\end{equation*}	
and we define the graph norm of $f\in \dom(\mathbb{A})$ by 
\begin{equation*}
	\|f\|_{\mathbb{A}} := \|f\|_{E_0} + \|\partial_x^2 f\|_{E_0} + \|af\|_{E_0}\,.
\end{equation*} 
The operator $\mathbb{A}$ includes the linear terms on the right-hand side of \eqref{FD.1} (diffusion + fragmentation) and generates an analytic semigroup in $E_0$ \cite{LaWa2021}. We shall recall its properties later, wherever needed. Setting 
\begin{equation*}
	E_1 := (\dom(\mathbb{A}),\|\cdot\|_{\mathbb{A}})\,,
\end{equation*} 
we finally introduce the complex interpolation spaces
\begin{equation}
	E_\xi := \big[E_0,E_1\big]_{\xi}\,,\qquad \xi\in (0,1)\,, \label{IS}
\end{equation}
and denote the corresponding norm by $\|\cdot\|_{E_\xi}$. The positive cone $E_\xi^+$ of $E_\xi$ is then
\begin{equation*}
	E_\xi^+ := \{ f\in E_\xi\ :\ f\ge 0 \;\;\text{a.e. in}\; (0,\infty)\}\,.
\end{equation*}

\bigskip

Now, the coagulation-fragmentation equation with size diffusion \eqref{FD.0} can be reformulated as the semilinear Cauchy problem
\begin{equation}\label{CP}
\frac{\mathrm{d}\phi}{\mathrm{d}t}(t)=\mathbb{A}\phi(t)+\mathcal{K}(\phi(t))\,,\quad t>0\,,\qquad \phi(0)=f\,,
\end{equation}
and we first prove that \eqref{CP} is locally well-posed in $E_\xi$ for appropriate values of $\xi$.

%%%%%%%%%%%%%%%%
\begin{theorem}\label{th1}
	Suppose that the rate coefficients $a$, $b$, and $k$ satisfy \eqref{A.0}, \eqref{B.0}, \eqref{B.10}, and \eqref{K1} with parameters $0\le \theta_0 <\theta< 1$ and $m>1$, and consider $\xi\in [0,1)$ with $2\theta<1+\xi$.
	
	\medskip
	
	\noindent{\bf (a) Local Existence:} Given $f\in E_\xi$,  the Cauchy problem~\eqref{CP} has a unique maximal strong solution 
	\begin{equation*}
		\phi = \phi(\cdot;f)\in C\big([0,t^+(f)),E_\xi\big)\cap C\big((0,t^+(f)),E_1\big)\cap C^1\big((0,t^+(f)),E_0\big)
	\end{equation*}
	with $t^+(f)\in (0,\infty]$ such that, if $\xi<\theta$, then
	\begin{equation*}
		\lim_{t\to 0} t^{\theta-\xi}\|\phi(t)\|_{E_\theta}=0\,.
	\end{equation*} 
	Moreover,
	\begin{equation}
		M_1\left(\phi(t) \right) = M_1(f)\,, \qquad t\in [0,t^+(f))\,. \label{M.102}
	\end{equation} 
	\medskip
	
	\noindent{\bf (b) Continuous Dependence:} The mapping $[(t,f)\mapsto \phi(t;f)]$ defines a semiflow on $E_\xi$.\medskip
	
	\noindent{\bf (c) Positivity:} If $f\in E_\xi^+$, then $\phi(t;f)\ge 0$ for $t\in [0,t^+(f))$. \medskip
	
	\noindent{\bf (d) Global Existence Criterion:} Let $f\in E_\xi$. If there are $0<t_0<T\wedge t^+(f)$ such that
	\begin{equation}\label{global}
		\sup_{t_0<t< T\wedge t^+(f)}\| \phi(t)\|_{Y}<\infty\,,
	\end{equation}
	then  $t^+(f)\ge T$.
	In particular, $t^+(f)=\infty$ if \eqref{global} holds true for arbitrary $0<t_0<T$. Furthermore, if there is $K_*>0$ such that
	\begin{equation}\label{K1b}
		0\le k(x,y)\le K_*\frac{\ell(x) (y+y^m)+ \ell(y) (x+x^m)}{x+y+(x+y)^m}\,,\qquad (x,y)\in (0,\infty)^2\,,
	\end{equation}
	and
	\begin{equation}\label{global1bb}
		\sup_{0<t<T\wedge t^+(f)}\| \phi(t)\|_{E_0}<\infty\,,\qquad T>0\,,
	\end{equation}
	then $t^+(f)=\infty$.
\end{theorem}
%%%%%%%%%%%%%%%%

A striking difference between \Cref{th1} and similar results on coagulation-fragmentation equations (possibly with growth or decay, but without size diffusion) is the underlying functional framework \cite{AiBa1979, BaLa2012, BaLa2020, BLL2020b, MLM1997}. Indeed, coping with the coagulation term $\mathcal{K}(\phi)$ requires in general that $\phi$ belongs to $L_1((0,\infty),(1+x^m)\mathrm{d}x)$, a property which is provided here by the diffusion when starting from the smaller space $E_0$.

%%%%%%%%%%%%%%%%
\begin{remark}\label{rem1}
	Unfortunately, we do not have a precise characterization of the space $E_\xi$ for $\xi\in (0,1)$, besides some embedding in a weighted $L_1$-space as in \Cref{L1} below. Nevertheless, if $\theta\in (0,1/2)$, then we can take $\xi=0$ in \Cref{th1}, which then provides the well-posedness of \eqref{FD.0} in $E_0=L_1((0,\infty),(x+x^m)\mathrm{d}x)$. In the same vein, the global existence criterion~\eqref{global1bb} only requires an a priori estimate in the norm of $E_0$.
\end{remark}
%%%%%%%%%%%%%%%%

The proof of \Cref{th1} relies on a more or less standard approach, involving time-weighted spaces, see \cite{Wei1980}, \cite[Chapter~4]{Yag2010} (and, e.g.,  \cite{AmWa2005} for an application of this approach to the coagulation-fragmentation equation with spatial diffusion). It strongly relies on the above mentioned fact that the operator $\mathbb{A}$ generates a positive analytic semigroup in $E_0$, see \cite[Theorem~1.1]{LaWa2021} and \Cref{P1} below, and the Lipschitz continuity of the coagulation operator $\mathcal{K}$ in $Y$ stemming from either \eqref{K1} or \eqref{K1b}. This approach is adapted to the present situation and allows one to consider initial values in $E_\xi$ for $\xi<\theta$, bearing in mind that the coagulation term is not necessarily well-defined in that space. We provide the proof of \Cref{th1} in \Cref{sec2}.

\bigskip

Let us collect a few comments on the assumptions \eqref{K1} and \eqref{K1b} on the coagulation kernel $k$.
	
\medskip
	
\noindent{\bf (i)} If there are $\alpha\in (1/2,1]$ and $K>0$ such that
\begin{equation}
	k(x,y) \le K (1+a(x))^\alpha (1+a(y))^\alpha\,, \qquad (x,y)\in (0,\infty)^2\,, \label{K2}
\end{equation}
then $k$ satisfies \eqref{K1} with $\theta_0=1/2$, $\theta=\alpha$, and any $m>1$. The assumption~\eqref{K2} is introduced in \cite{BaLa2012, Ban2020} to study the local well-posedness of the coagulation-fragmentation equation (without diffusion), see also \cite{BLL2020b}.
	
\medskip
	
\noindent{\bf (ii)} Let $k$ be a coagulation kernel satisfying \eqref{K1} with parameters $(\theta_0,\theta,m)$. Then it also satisfies \eqref{K1} with parameters $(\theta_0,\theta,m')$ for any $m'>m$.
	
\medskip

\noindent{\bf (iii)} According to \eqref{K1}, the coagulation kernel $k$ may feature a singularity for small sizes. In particular, given $-1<\alpha \le 0 \le \beta < 1$, the choice
\begin{equation*}
	k(x,y) = \left( x^{\alpha} + y^{\alpha} \right) \left( x^\beta + y^\beta \right)\,, \qquad (x,y)\in (0,\infty)^2\,,
\end{equation*}
complies with \eqref{K1} for $\theta_0 = (1-\alpha)/2 \ge 1/2$, provided that $\big( y\mapsto y^\beta (1+a(y))^{-\theta} \big)$ is bounded on $(1,\infty)$, the parameter $m>1$ being arbitrary. This example includes Smoluchowski's celebrated coagulation kernel corresponding to $\beta = -\alpha = 1/3$ \cite{Smo1916, Smo1917}. 

\medskip

\noindent{\bf (iv)} It is easy to check that \eqref{K1b} implies \eqref{K1} with $k_*=2K_*$. 

\medskip

\noindent{\bf (v)} When $\theta_0\in [0,1/2]$, the assumption \eqref{K1b} implies that $k\in L_\infty((0,1)\times (1,\infty))$. This excludes coagulation kernels satisfying 
\begin{equation}
	k(x,y) = K \left[ 1+a(x))^\alpha  + (1+a(y))^\alpha \right]\,, \qquad (x,y)\in (0,\infty)^2\,. \label{K4}
\end{equation}
Such an assumption is known to guarantee global existence of classical solutions for the case without size diffusion, see \cite{Ban2020, BLL2020b}.

\bigskip

Building upon \Cref{th1}~{\bf (d)}, we supplement \Cref{th1} with a global existence result for a specific class of coagulation kernels.

%%%%%%%%%%%%%%%%
\begin{theorem}[{\bf Global Existence}]\label{th2}
	Suppose that the rate coefficients $a$, $b$, and $k$ satisfy \eqref{A.0}, \eqref{B.0}, \eqref{B.10}, and \eqref{K1b} with parameters $0\le \theta_0 <\theta< 1$ and $m>1$ and consider $\xi\in [0,1)$ with $2\theta<1+\xi$. Assume further that $\theta_0\in [0,1/2]$ and that there is $k_0>0$ such that
	\begin{equation}
		k(x,y) \le k_0 \frac{xy}{x+y} \left[ (1+a(x))^\theta + (1+a(y))^\theta \right]\,, \qquad (x,y)\in (1,\infty)^2\,. \label{K3}
	\end{equation}
	Then $t^+(f)=\infty$ for all $f\in E_\xi^+$.
\end{theorem}
%%%%%%%%%%%%%%%%

Observe that \eqref{K3} is stronger than \eqref{K1b} for $(x,y)\in (1,\infty)^2$ (and somehow corresponds to taking $m=1$ in \eqref{K1b} and the definition~\eqref{L} of $\ell$).  Guided by \eqref{global1bb}, the proof of Theorem~\ref{th2} relies on an a priori estimate in $E_0$, which is derived in Section~\ref{sec3}.

\bigskip

From now on, we assume that the rate coefficients $a$, $b$, and $k$ satisfy \eqref{A.0}, \eqref{B.0}, \eqref{B.10}, and \eqref{K1} with fixed parameters $0\le \theta_0<\theta\le 1$ and $m>1$.

%%%%%%%%%%%%%%%%
%%%%%%%%%%%%%%%%
\section{Well-posedness}\label{sec2}
%%%%%%%%%%%%%%%%
%%%%%%%%%%%%%%%%

%%%%%%%%%%%%%%%%
%%%%%%%%%%%%%%%%
\subsection{Diffusion \& Fragmentation}\label{sec2.1}
%%%%%%%%%%%%%%%%
%%%%%%%%%%%%%%%%

In view of \eqref{CP} and according to the above discussion, we begin with the linear terms on the right-hand side of \eqref{FD.1} and first recall the generation properties of $\mathbb{A}$ in $E_0$ established in \cite{LaWa2021}.

%%%%%%%%%%%%%%%%
\begin{proposition}\label{P1}
	The operator $\mathbb{A}$ generates a positive analytic semigroup $(U(t))_{t\ge 0}$ on $E_0$ with
	\begin{equation}
		M_1\left(U(t)f \right) = M_1(f)\,, \qquad t\ge 0\,,\quad f\in E_0\,. \label{M.100}
	\end{equation}
In addition, $(U(t))_{t\ge 0}$ is an analytic semigroup on $E_\xi$ for every $\xi\in (0,1)$.
\end{proposition}
%%%%%%%%%%%%%%%%

\begin{proof} 
	The first statement in \Cref{P1} follows from \cite[Theorem~1.1]{LaWa2021} and implies the second one, according to \cite[Theorem~6]{Ama1986} or \cite[II.Theorem~2.1.3]{Ama1995}.
\end{proof}

As already mentioned, the functional framework we shall work with involves the interpolation spaces $(E_\xi)_{\xi\in (0,1)}$, for which a simple characterization is not obvious to derive. Nevertheless, we identify a weighted $L_1$-space in which $E_\theta$ embeds continuously.

%%%%%%%%%%%%%%%%
\begin{lemma}\label{L1}
	The embedding of the interpolation space $E_\theta$ in the weighted $L_1$-space $Y$ defined in \eqref{SY} is continuous.
\end{lemma}
%%%%%%%%%%%%%%%%
 
\begin{proof} 
Since 
\begin{equation*}
 	E_1\hookrightarrow L_1\big( (0,\infty),(1+a(x))(x+x^m)\mathrm{d}x \big) 
\end{equation*}
and
\begin{align*}
	& \big[L_1\big( (0,\infty),(x+x^m)\mathrm{d}x \big), L_1\big( (0,\infty),(1+a(x))(x+x^m)\mathrm{d}x \big)\big]_{\theta} \\
	& \hspace{4cm} \doteq L_1\big( (0,\infty),(1+a(x))^\theta (x+x^m)\mathrm{d}x \big)\,,
\end{align*}
it readily follows that
\begin{equation*}
 	E_\theta = \big[ E_0,E_1 \big]_\theta \hookrightarrow L_1\big( (0,\infty),(1+a(x))^\theta (x+x^m)\mathrm{d}x \big)\,. 
\end{equation*}
Next, recall from \cite[Lemma~2.1]{LaWa2021} that $E_1\hookrightarrow X_r$ for $r\in (-1,1)$. Therefore,
\begin{equation*}
 	E_\theta\hookrightarrow \big[X_1, X_r\big]_{\theta} = X_{1+\theta(r-1)}\,,\qquad r\in (-1,1)\,,
\end{equation*}
and the choice $r=1-(2\theta_0/\theta)\in (-1,1)$ completes the proof.
\end{proof}

We next derive some positivity properties for the semigroups generated by specific perturbations of $\mathbb{A}$ on $E_0$, which are required later on to establish the non-negativity of solutions to \eqref{FD.0}.

%%%%%%%%%%%%%%%%
\begin{lemma}\label{L4}
Let $\gamma>0$ and 
\begin{equation*}
	V(x):=\gamma\frac{\ell(x)}{x+x^m}\,,\quad x>0\,,
\end{equation*}
where $\ell$ is defined in \eqref{L}. Then $\mathbb{A}_V:=\mathbb{A}-V$ with $\mathrm{dom}(\mathbb{A}_V) = \mathrm{dom}(\mathbb{A})$ generates a positive analytic semigroup on $E_0$.
\end{lemma}
%%%%%%%%%%%%%%%%

\begin{proof} Note that $V:=(f\mapsto Vf)\in \mathcal{L}(E_\theta,E_0)$ due to \Cref{L1}. Since $V$ is thus $\mathbb{A}$-bounded with a zero $\mathbb{A}$-bound and $\mathbb{A}$ generates an analytic semigroup on $E_0$, it follows from \cite[Theorem~III.2.10]{EnNa2000} that $\mathbb{A}_V$ with $\mathrm{dom}(\mathbb{A}_V)=\mathrm{dom}(\mathbb{A})$ generates an analytic semigroup on $E_0$. 
	
It remains to prove that this semigroup is positive. To this end, pick $\lambda>0$ sufficiently large and $g\in E_0$ with $g\le 0$. Then there is $u\in \mathrm{dom}(\mathbb{A})$ such that
\begin{equation*}
	(\lambda-\mathbb{A}+V)u=g\,.
\end{equation*}
Multiplying this identity by $x\,\mathrm{sign}_+(u(x))$ and integrating over $(0,\infty)$ give 
\begin{align*}
	\int_0^\infty x\, \big(\lambda+V(x)\big) u_+(x)\, \rd x &-\int_0^\infty x\,\mathrm{sign}_+(u(x))\,\partial_x^2 u(x)\, \rd x\\
	&=\int_0^\infty x\,\mathrm{sign}_+(u(x))\, \mathcal{F}(u)(x)\, \rd x + \int_0^\infty x\, g (x) \,\mathrm{sign}_+(u(x))\, \rd x\,.
\end{align*}
On the one hand, using \eqref{B.0} and Fubini's theorem,
\begin{align*}
	\int_0^\infty x\,\mathrm{sign}_+(u(x)) \mathcal{F}(u)(x)\,\rd x &=\int_0^\infty x\int_x^\infty a(y)\, b(x,y)\, \mathrm{sign}_+(u(x))\, u(y)\,\rd y\rd x \\
	& \qquad -\int_0^\infty x\, a(x)\, u_+(x) \, \rd x\\
	&\le \int_0^\infty a(y)\,  u_+(y) \int_0^y x\, b(x,y)\,\rd x\rd y - \int_0^\infty x\, a(x)\,  u_+(x)\, \rd x\le 0\,.
\end{align*}
On the other hand,
\begin{align*}
	-\int_0^\infty x\,\mathrm{sign}_+(u(x))\,\partial_x^2 u(x)\,\rd x \ge \int_0^\infty \partial_x u_+(x)\,\rd x =0
\end{align*}
by Kato's inequality \cite[Lemma~A]{Kat1972} and the boundary condition $u(0)=0$, the functions $u$ and $\partial_x u$ vanishing sufficiently rapidly at infinity by \cite[Lemma~2.1]{LaWa2021}. Collecting these inequalities and using the non-positivity of $g$, we end up with
\begin{equation*}
	\int_0^\infty x\,\big(\lambda+V(x)\big)\, u_+(x)\,  \rd x \le 0\,.
\end{equation*}
Consequently, $u\le 0$. We have thus shown that the operator $\mathbb{A}_V=\mathbb{A}-V$ is resolvent positive; that is, the semigroup generated by $\mathbb{A}_V$ is positive.
\end{proof}

%%%%%%%%%%%%%%%%
%%%%%%%%%%%%%%%%
\subsection{Coagulation}\label{sec2.2}
%%%%%%%%%%%%%%%%
%%%%%%%%%%%%%%%%

We next focus on the coagulation term on the right-hand side of \eqref{FD.1} and study the Lipschitz continuity of the coagulation operator $\mathcal{K}$, noting that it is bilinear. 

%%%%%%%%%%%%%%%%
\begin{lemma}\label{L2}
The operator $\mathcal{K}$ belongs to $\mathcal{L}^2(Y, E_0)$, the space $Y$ being defined in \eqref{SY}, and
\begin{equation}
	M_1(\mathcal{K}(\phi)) = 0\,, \qquad \phi\in Y\,. \label{CM}
\end{equation}
Furthermore, if $k$ satisfies additionally \eqref{K1b}, then 
\begin{equation}\label{lp}
	\|\mathcal{K}(\psi,\phi)\|_{E_0} \le 
	\frac{3K_*}{2} \big( \|\psi\|_{Y}\, \|\phi\|_{E_0} + \|\psi\|_{E_0}\, \|\phi\|_{Y}\big)\,, \qquad (\psi,\phi)\in Y^2\,.
\end{equation}
\end{lemma}
%%%%%%%%%%%%%%%%

\begin{proof}
Given $(\psi,\phi)\in Y^2$,  Fubini's theorem and \eqref{K1} imply that
\begin{align}
	\|\mathcal{K}(\psi,\phi)\|_{E_0} &\le \frac{1}{2} \int_0^\infty (x+x^m) \int_0^x k(y,x-y)\, |\psi(y)|\,|\phi(x-y)|\,\mathrm{d}y\mathrm{d}x \nonumber\\
	&\qquad +\int_0^\infty (x+x^m)\,|\psi(x)| \int_0^\infty k(x,y) |\phi(y)|\, \mathrm{d}y\mathrm{d}x \nonumber\\
	&\le \frac{3}{2}\int_0^\infty\int_0^\infty k(x,y)\,\big[ x+y+(x+y)^m \big]\, \vert\psi(x)\vert\,\vert\phi(y)\vert\,\mathrm{d}y\mathrm{d}x \label{lll} \\
	&\le \frac{3k_*}{2} \left( \int_0^\infty \ell(x) |\psi(x)|\, \mathrm{d}x \right) \left( \int_0^\infty \ell(y) |\phi(y)|\, \mathrm{d}y \right) \nonumber\\
	&\le \frac{3k_*}{2}\|\psi\|_{Y}\,\|\phi\|_{Y}\,, \nonumber
	\end{align}
which proves the claim $\mathcal{K}\in\mathcal{L}^2(Y, E_0)$. A classical computation based on Fubini's theorem then gives~\eqref{CM}.

\medskip

Assume next that $k$ satisfies additionally \eqref{K1b}. Then, for $(\psi,\phi)\in Y^2$, the estimate \eqref{lll} entails
\begin{align*}
	\|\mathcal{K}(\psi,\phi)\|_{E_0} &\le \frac{3}{2} \int_0^\infty\int_0^\infty k(x,y)\, \big[ x+y+(x+y)^m \big]\, |\psi(x)|\, |\phi(y)|\, \mathrm{d}y \mathrm{d}x  \\
	&\le \frac{3K_*}{2} \int_0^\infty\int_0^\infty \big[ \ell(x) (y+y^m) + \ell(y) (x+x^m) \big]\, |\psi(x)|\, |\phi(y)|\, \mathrm{d}y \mathrm{d}x  \\
	&\le \frac{3K_*}{2} \big( \|\psi\|_{Y}\, \|\phi\|_{E_0} + \|\psi\|_{E_0}\, \|\phi\|_{Y} \big)\,,
	\end{align*}
which completes the proof of  \Cref{L2}.	
\end{proof}

%%%%%%%%%%%%%%%%
%%%%%%%%%%%%%%%%
\subsection{Proof of \Cref{th1}}\label{sec2.3}
%%%%%%%%%%%%%%%%
%%%%%%%%%%%%%%%%

Having established the above preliminary results, we are now in a position to begin the proof of  \Cref{th1}. Since we aim at handling initial values with mild regularity, we introduce the following time-weighted spaces.
Given $T>0$, a Banach space $E$, and $\mu\in \mathbb{R}$, we denote the Banach space of all functions
$u\in C((0,T],E)$ such that $\big(t\mapsto t^{\mu}u(t)\big)$ is bounded on $(0,T]$ by $BC_{\mu}((0,T],E)$, equipped with the norm
\begin{equation*}
u\mapsto \|u\|_{BC_{\mu}((0,T],E)} := \sup_{t\in (0,T]}\left\{ t^{\mu}\, \|u(t)\|_E \right\}\,.
\end{equation*}
We write $C_{\mu}((0,T],E)$ for the closed linear subspace thereof consisting of all $u\in BC_{\mu}((0,T],E)$ satisfying additionally $t^{\mu} u(t)\rightarrow 0$ in $E$ as $t\rightarrow 0$. Note that $BC_{\nu}((0,T],E)\hookrightarrow BC_{\mu}((0,T],E)$ for $\nu\le \mu$ with
\begin{equation}
	\|u\|_{BC_{\mu}((0,T],E)} \le T^{\mu-\nu} \|u\|_{BC_{\nu}((0,T],E)}\,, \qquad u\in BC_{\nu}((0,T],E)\,. \label{X1}
\end{equation}
We denote the open and closed balls of $E$ centered at $f\in E$ and of radius $R$ by $\mathbb{B}_E(f,R)$ and $\bar{\mathbb{B}}_E(f,R)$, respectively.

To start with, we analyze how the semigroup $(U(t))_{t\ge 0}$ generated by $\mathbb{A}$ on $E_0$, see \Cref{P1}, acts on $C_{\mu}((0,T],E_\theta)$. The subsequent \Cref{L3} and \Cref{L6} are more or less implicitly contained in the proof of \cite[Theorem~4.1]{Yag2010} with the difference that domains of fractional powers are used instead of interpolation spaces. 

%%%%%%%%%%%%%%%%
\begin{lemma}\label{L3}
Let $0\le \xi<\eta\le 1$ with $(\xi,\eta)\not=(0,1)$ and $0<T\le T_0$. There is $c_0(T_0,\xi,\eta)>0$ such that, if $f\in E_\xi$, then
\begin{equation*}
	Uf:=\big( t\mapsto U(t)f \big)\in C_{\eta-\xi}((0,T],E_\eta)\quad \text{with }\quad \|Uf\|_{C_{\eta-\xi}((0,T],E_\eta)}\le c_0(T_0,\xi,\eta)\|f\|_{E_\xi} \,.
\end{equation*}
\end{lemma}
%%%%%%%%%%%%%%%%

\begin{proof}
Recall that there is $c_0(T_0,\xi,\eta)>0$ such that
\begin{equation}\label{S}
	t^{\eta-\xi}\|U(t)\|_{\mathcal{L}(E_\xi,E_\eta)} + \|U(t)\|_{\mathcal{L}(E_\eta)} + t\|\mathbb{A}U(t)\|_{\mathcal{L}(E_\eta)}\le c_0(T_0,\xi,\eta)\,,\qquad t\in (0,T_0]\,,
\end{equation}
see \cite[II.Lemma~5.1.3]{Ama1995}. Thus, since $(U(t))_{t\ge 0}$ is a strongly continuous semigroup on $E_\eta$, see \Cref{P1}, it readily follows from \eqref{S} that 
\begin{equation*}
	\big( t\mapsto U(t)f\big)\in BC_{\eta-\xi}((0,T],E_\eta) \quad \text{with }\quad \|Uf\|_{BC_{\eta-\xi}((0,T],E_\eta)}\le c_0(T_0,\xi,\eta)\|f\|_{E_\xi} 
\end{equation*}
for all $f\in E_\xi$ and $T\in (0,T_0]$. 

We next recall that $E_\eta$ is dense in $E_\xi$  since $[\cdot,\cdot]_\xi$ is an admissible interpolation functor. Hence, given $\varepsilon >0$, there is $g\in E_\eta$ such that 
\begin{equation*}
	\|f-g\|_{E_\xi}\le \frac{\varepsilon}{c_0(T_0,\xi,\eta)}\,.
\end{equation*}
Therefore, by \eqref{S},
\begin{align*}
	t^{\eta-\xi}\|U(t)f\|_{E_\eta} & \le t^{\eta-\xi} \|U(t)(f-g)\|_{E_\eta} + t^{\eta-\xi} \|U(t) g\|_{E_\eta} \\
	& \le t^{\eta-\xi} \|U(t)\|_{\mathcal{L}(E_\xi,E_\eta)} \|f-g\|_{E_\xi} + t^{\eta-\xi}\|U(t)\|_{\mathcal{L}(E_\eta)}\| g\|_{E_\eta} \\
	& \le \varepsilon + c_0(T_0,\xi,\eta) \|g\|_{E_\eta} t^{\eta-\xi}\,,
\end{align*}
so that, since $\eta>\xi$,
\begin{equation*}
	\limsup_{t\to 0} t^{\eta-\xi}\|U(t)f\|_{E_\eta} \le \varepsilon\,.
\end{equation*}
We then let $\varepsilon\to 0$ to complete the proof.
\end{proof}

The next step is to elucidate the behavior of $(U(t))_{t\ge 0}$ when involved in a convolution with respect to time. To this end, given $T>0$  and $u:(0,T]\rightarrow E_0$, we set
\begin{equation*}
	U\star u (t):=\int_0^t U(t-s) u(s)\,\rd s\,,\qquad t\in (0,T]\,,
\end{equation*}
whenever this integral makes sense.

%%%%%%%%%%%%%%%%
\begin{lemma}\label{L6}
Consider $(2\eta,\nu)\in [0,1)^2$ and $0<T\le T_0$. Then
\begin{equation}\label{l2a}
	[(u,v)\mapsto U\star \mathcal{K}(u,v)]\in \mathcal{L}^2\big(C_{\eta}((0,T],E_\theta),C_{2\eta+\nu-1}((0,T],E_\nu)\big)
\end{equation}
and
\begin{equation}\label{l2b}
	[(u,v)\mapsto U\star \mathcal{K}(u,v)]\in  \mathcal{L}^2\big(BC((0,T],E_\theta),  BC_{\nu-1}((0,T],E_\nu)\big)\,,
\end{equation}
where $\theta$ is defined in \eqref{K1L} and the norms of the above bilinear form depend only on $T_0$, $\eta$, and $\nu$.
\end{lemma}
%%%%%%%%%%%%%%%%

\begin{proof}
Pick $\mu\in [0,1)$ and $u\in BC_{\mu}((0,T],E_0)$. Let $t\in (0,T]$. By \eqref{S},
\begin{align}
	\|U\star u(t)\|_{E_\nu} &\le \int_0^t \|U(t-s)\|_{\mathcal{L}(E_0,E_\nu)}\, \| u(s)\|_{E_0}\,\rd s  \nonumber\\
	& \le c_0(T_0,0,\nu) \int_0^t (t-s)^{-\nu} s^{-\mu}\,\rd s\, \|u\|_{BC_ {\mu}((0,t],E_0)}\nonumber\\
	& = c_0(T_0,0,\nu) t^{1-\nu-\mu}\, \mathsf{B}(1-\nu,1-\mu)\,\|u\|_{BC_{\mu}((0,t],E_0)}\,,\label{X2}
\end{align}
where $\mathsf{B}$ denotes the Beta function. Therefore, since
\begin{equation*}
	\|u\|_{BC_{\mu}((0,t],E_0)} \le \|u\|_{BC_{\mu}((0,T],E_0)}\,,
\end{equation*}
we infer from \eqref{X2} that 
\begin{equation}
	[u\mapsto U\star u]\in \mathcal{L}\big(BC_{\mu}((0,T],E_0), BC_{\mu+\nu-1}((0,T],E_\nu)\big)\,. \label{X3}
\end{equation}
Assume then that $u\in C_\mu((0,T],E_0)$. This property, along with \eqref{X2}, readily implies that $U\star u$ belongs to $C_{\mu+\nu-1}((0,T],E_\nu)$ and we conclude that
\begin{equation}
	[u\mapsto U\star u]\in \mathcal{L}\big(C_{\mu}((0,T],E_0), C_{\mu+\nu-1}((0,T],E_\nu)\big)\,. \label{X4}
\end{equation}

Now, if $(u,v)\in \big( BC_\rho((0,T],E_\theta) \big)^2$ for some $\rho\in\mathbb{R}$, then \Cref{L1} and \Cref{L2} entail that, for $t\in (0,T]$,
\begin{align*}
	t^{2\rho} \|\mathcal{K}(u(t),v(t))\|_{E_0} & \le t^{2\rho} \|\mathcal{K}\|_{\mathcal{L}^2(Y,E_0)} \|u(t)\|_Y \|v(t)\|_Y \\
	& \le c_1^2 \|\mathcal{K}\|_{\mathcal{L}^2(Y,E_0)} \|u\|_{BC_\rho((0,t],E_\theta)} \|v\|_{BC_\rho((0,t],E_\theta)}\,,
\end{align*}
where $c_1$ is the norm of the continuous embedding of $E_\theta$ in $Y$, see \Cref{L1}. Consequently,
\begin{equation}
	\mathcal{K}(u,v)\in BC_{2\rho}((0,T],E_0)\,. \label{X5}
\end{equation}
Also, if $u\in C_\rho((0,T],E_\theta)$ or $v\in C_\rho((0,T],E_\theta)$, then
\begin{equation}
	\mathcal{K}(u,v)\in C_{2\rho}((0,T],E_0)\,. \label{X6}
\end{equation}
We then combine \eqref{X4} (with $\mu=2\eta$) and \eqref{X6} (with $\rho=\eta$) to derive~\eqref{l2a}, while~\eqref{l2b} follows from~\eqref{X3} (with $\mu=0$) and~\eqref{X5} (with $\rho=0$).
\end{proof}

We are now ready to provide the proof of  \Cref{th1}.

\begin{proof}[Proof of \Cref{th1}]
{\bf Step~1: Local Existence.} Let $T_0>0$ be arbitrary. We consider $\xi\in [0,1)$ with $2\theta<1+\xi$ and $f\in E_\xi$, but handle the ranges $\xi\in [0,\theta)$ and $\xi\in [\theta,1)$ differently. In the following, $c$ and $(c_i)_{i\ge 2}$ denote positive constants depending only on $\theta_0$, $\theta$, $m$, $T_0$, and $\xi$. Dependence upon additional parameters will be indicated explicitly.

\medskip

Assume first that $\xi\in [0,\theta)$ and let $T\in (0,T_0]$. We note that \eqref{l2a} (with $2\eta=2(\theta-\xi)<1$ and $\nu=\theta$)  implies that,
\begin{equation}\label{l1}
	[(u,v)\mapsto U\star \mathcal{K}(u,v)]\in \mathcal{L}^2\big(C_{\theta-\xi}((0,T],E_\theta), C_{3\theta-2\xi-1}((0,T],E_\theta)\big)\,,
\end{equation}
with a norm depending only on $\theta$, $T_0$, and $\xi$, while the constraint $2\theta<1+\xi$ and \eqref{X1} guarantee
\begin{equation}\label{l1b}
		C_{3\theta-2\xi-1}((0,T],E_\theta)\hookrightarrow C_{\theta-\xi}((0,T],E_\theta)
\end{equation}
with norm bounded by $T^{1+\xi-2\theta}$. Introducing 
\begin{equation*}
	Z_T := C_{\theta-\xi}((0,T],E_\theta)\,,
\end{equation*}	 
it follows from \eqref{l1} and \eqref{l1b} that there is $c_2>0$ such that 
\begin{equation}
	\| U\star \mathcal{K}(\phi)\|_{Z_T}\le c_2\,  T^{1+\xi-2\theta}\, \|\phi\|_{Z_T}^2\,, \qquad \phi\in Z_T\,, \label{X10}
\end{equation}
and
\begin{equation}
	\| U\star \mathcal{K}(\phi)-U\star \mathcal{K}(\psi)\|_{Z_T}\le c_2\,  T^{1+\xi-2\theta}\, \big(\|\phi\|_{Z_T}+\|\psi\|_{Z_T}\big)\, \|\phi-\psi\|_{Z_T}\,, \qquad (\phi, \psi)\in Z_T^2\,. \label{X11}
\end{equation}
Given $R\ge \|f\|_{E_\xi}$, we choose $T_1=T_1(R)\in (0,T_0]$ such that 
\begin{equation*}
	4 \big( 1 + c_0(T_0,\xi,\theta) R \big)^2 c_2 T_1^{1+\xi-2\theta}\le 1\,,
\end{equation*}
the constant $c_0(T_0,\xi,\theta)$ being defined in \Cref{L3}, and infer from \Cref{L3} (with $\eta=\theta$) and the choice of $R$ that
\begin{equation}
	\|Uf\|_{Z_{T_1}}\le c_0(T_0,\xi,\theta)R\,. \label{X12}
\end{equation}	 
We now define 
\begin{equation*}
	F(\phi) := Uf +  U\star \mathcal{K}(\phi)\,, \qquad \phi\in \bar{\mathbb{B}}_{Z_{T_1}}(Uf,1)\,.
\end{equation*}
On the one hand, it follows from \eqref{X10} and \eqref{X12} that $F(\phi)\in Z_{T_1}$ and, by the choice of $T_1$, 
\begin{align*}
	\|F(\phi) - Uf\|_{Z_{T_1}} & \le \|U\star \mathcal{K}(\phi)\|_{Z_{T_1}} \le c_2 T_1^{1+\xi-2\theta} \|\phi - Uf + Uf\|_{Z_{T_1}}^2 \\
	& \le c_2 T_1^{1+\xi-2\theta} \big( 1 + c_0(T_0,\xi,\theta)R \big)^2 \le 1\,,
\end{align*}
so that $F(\phi)\in \bar{\mathbb{B}}_{Z_{T_1}}(Uf,1)$. On the other hand, for $(\phi,\psi)\in \big( \bar{\mathbb{B}}_{Z_{T_1}}(Uf,1) \big)^2$, we deduce from \eqref{X11} that
\begin{align}
	\|F(\phi)-F(\psi)\|_{Z_{T_1}} & = \|U\star \mathcal{K}(\phi)-U\star \mathcal{K}(\psi)\|_{Z_{T_1}}\nonumber\\
	& \le c_2 T_1^{1+\xi-2\theta} \big( \|\phi\|_{Z_{T_1}} + \|\psi\|_{Z_{T_1}} \big) \|\phi-\psi\|_{Z_{T_1}} \nonumber\\
	& \le  2 c_2 T_1^{1+\xi-2\theta} \big( 1 + c_0(T_0,\xi,\theta)R \big) \|\phi-\psi\|_{Z_{T_1}} \nonumber\\
	& \le \frac{1}{2} \|\phi-\psi\|_{Z_{T_1}}\,.\label{wq}
\end{align}
Consequently, $F$ is a strict contraction on $\bar{\mathbb{B}}_{Z_{T_1}}(Uf,1)$, so that Banach's fixed point theorem yields a unique fixed point $\phi=\phi(\cdot;f) \in \bar{\mathbb{B}}_{Z_{T_1}}(Uf,1)$ of $F$; that is, for $t\in (0,T_1]$,
\begin{equation}\label{phi}
	\phi(t) = \left(Uf +  U\star \mathcal{K}(\phi)\right)(t) =  U(t)f +  \int_0^t U(t-s)\,\mathcal{K}(\phi)(s)\,\rd s
\end{equation}
and $\phi$ is thus a mild solution to \eqref{CP} in $E_\theta$ on $(0,T_1]$. From \eqref{X1}, \eqref{l2a} (with $\eta=\theta-\xi$ and $\nu=\xi$),  and the constraint $2\theta < 1+\xi$, we deduce that
\begin{equation*}
	U\star \mathcal{K}(\phi)\in C_{2\theta-\xi-1}((0,T_1],E_\xi)\hookrightarrow C_{0}((0,T_1],E_\xi)\,,
\end{equation*}
which -- together with \eqref{phi} and the continuity property $Uf\in C([0,T_1],E_\xi)$ due to $f\in E_\xi$ and \Cref{P1} -- ensures that $\phi\in C([0,T_1],E_\xi)$. Moreover, for $(f_1,f_2)\in \big( \mathbb{B}_{E_\xi}(0,R) \big)^2$, the integral formulation \eqref{phi}, along with \Cref{L3} (with $\eta=\theta$) and \eqref{wq}, yields
\begin{align*}
	\|\phi(\cdot;f_1)-\phi(\cdot;f_2)\|_{Z_{T_1}} &\le \| U(f_1-f_2)\|_{Z_{T_1}} + \| U\star \mathcal{K}(\phi(\cdot;f_1))-U\star \mathcal{K}(\phi(\cdot;f_2))\|_{Z_{T_1}}\\
	& \le c_0(T_0,\xi,\theta)\|f_1-f_2\|_{E_\xi} +\frac{1}{2} \|\phi(\cdot;f_1) - \phi(\cdot;f_2)\|_{Z_{T_1}}\,,
\end{align*} 
and thus
\begin{equation}\label{p1}
	\|\phi(\cdot;f_1)-\phi(\cdot;f_2)\|_{Z_{T_1}} \le 2  c_0(T_0,\xi,\theta)\|f_1-f_2\|_{E_\xi}\,.
\end{equation}
Using again the embedding $C_{2\theta-\xi-1}((0,T_1],E_\xi)\hookrightarrow C_{0}((0,T_1],E_\xi)$, see \eqref{X1}, it  follows from \Cref{P1}, \eqref{l2a} (with $\eta=\theta-\xi$ and $\nu=\xi$), \eqref{phi}, and \eqref{p1} that
\begin{align*}
	& \|\phi(\cdot;f_1)-\phi(\cdot;f_2)\|_{C([0,T_1],E_\xi)} \\
	& \hspace{1cm} \le \| U(f_1-f_2)\|_{C([0,T_1],E_\xi)} + \| U\star \mathcal{K}(\phi(\cdot;f_1))-U\star \mathcal{K}(\phi(\cdot;f_2))\|_{C_{0}((0,T_1],E_\xi)}\\
	& \hspace{1cm} \le c \|f_1-f_2\|_{E_\xi} + c \big( \|\phi(\cdot;f_1)\|_{Z_{T_1}} + \|\phi(\cdot;f_2\|_{Z_{T_1}} \big) \|\phi(\cdot;f_1) - \phi(\cdot;f_2) \|_{Z_{T_1}} \\
	& \hspace{1cm} \le  c\|f_1-f_2\|_{E_\xi}\,.
\end{align*}
This estimate entails local uniqueness of solutions to \eqref{CP} and we can then extend $\phi(\cdot;f)$ to a unique mild solution on a maximal interval of existence $[0,t^+(f))$ by classical arguments.

Finally,  the conservation of mass stated in \eqref{M.102} is a consequence of \eqref{M.100}, \eqref{CM}, and \eqref{phi}.

\medskip
	
For the case $f\in E_\xi$ with $\xi\in [\theta,1)$, one  proceeds as above but  performs the fixed point argument for $F$ in $C([0,T_1],E_\xi)$ (instead of $Z_{T_1}$) with the help of \eqref{l2b} (instead of \eqref{l2a}).

\bigskip

\noindent{\bf Step~2: Regularity.} Given $\xi\in [0,1)$ with $2\theta<1+\xi$ and $f\in E_\xi$, we now check that $\phi=\phi(\cdot;f)$ is a classical solution to~\eqref{CP} in $E_0$ on $(0,t_+(f))$. To this end, let $\nu\in \big( \max\{\theta,\xi\},1 \big)$ and $T\in (0,t^+(f))$. If $\xi\in [0,\theta)$, then $\phi\in C_{\theta-\xi}((0,T],E_\theta)$. Hence, \Cref{L3} (with $\eta=\nu$) and \eqref{l2a} (with $\eta=\theta-\xi$) ensure that $Uf\in C_{\nu-\xi}((0,T],E_\nu)$ and $U\star \mathcal{K}(\phi) \in C_{2(\theta-\xi)+\nu-1}((0,T],E_\nu)$, respectively. Similarly, if $\xi\in [\theta,1)$, then $\phi\in C([0,T],E_\theta)$ and it follows from \Cref{L3} (with $\eta=\nu$) and from \eqref{l2b} that $Uf\in C_{\nu-\xi}((0,T],E_\nu)$ and $U\star \mathcal{K}(\phi) \in C_{\nu-1}((0,T],E_\nu)$, respectively. In both cases, since $T<t^+(f)$ is arbitrary, we then deduce from \eqref{phi} that 
\begin{equation*}
	\phi\in C\big( (0, t^+(f)),E_\nu \big)\,.
\end{equation*} 
Next, given $\varepsilon\in (0,t^+(f))$, we define $\phi_\varepsilon(\cdot) := \phi(\varepsilon+\cdot)$ and 
\begin{equation*}
	h_\varepsilon:=\mathcal{K}(\phi_\varepsilon)\in C([0,t^+(f)-\varepsilon),E_0)\,,
\end{equation*}
the latter being a consequence of \Cref{L1}, \Cref{L2}, and the regularity of $\phi$. According to \eqref{phi}, $\phi_\varepsilon$ is the unique mild solution in $E_0$ to the linear problem
\begin{equation*}
	\frac{\mathrm{d}\phi_\varepsilon}{\mathrm{d}t}(t)=\mathbb{A}\phi_\varepsilon(t)+h_\varepsilon(t)\,,\quad t\in [0,t^+(f)-\varepsilon)\,,\qquad \phi_\varepsilon(0)=\phi(\varepsilon)\in E_\nu\,,
\end{equation*}
so that $\phi_\varepsilon\in C^{\nu-\theta}([0,t^+(f)-\varepsilon),E_\theta)$ by \cite[II.Theorem~5.3.1]{Ama1995}. This last property, along with \Cref{L1} and \Cref{L2}, entails that
\begin{equation*}
	h_\varepsilon\in C^{\nu-\theta}([0,t^+(f)-\varepsilon),E_0)\,,
\end{equation*}
and \cite[II.Theorem~1.2.1]{Ama1995} implies that the mild solution $\phi_\varepsilon$ is actually a strong solution with
\begin{equation*}
	\phi_\varepsilon\in C^1((0,t^+(f)-\varepsilon),E_0)\cap C((0,t^+(f)-\varepsilon),E_1)\,.
\end{equation*}
As  $\varepsilon\in (0,t^+(f))$ is arbitrary,  we conclude that $\phi$ satisfies
\begin{equation*}
	\phi\in  C^1((0,t^+(f)),E_0)\cap C((0,t^+(f)),E_1)\,,
\end{equation*} 
and is thus a strong solution to \eqref{CP} on $(0,t^+(f))$.

\bigskip

\noindent{\bf Step~3: Continuous Dependence.} Let $f_0\in E_\xi$  and $t_0\in (0,t^+(f_0))$ be arbitrary. We fix $t_*\in (t_0,t^+(f_0))$ and $R>0$ such that $\phi([0,t_*];f_0)\subset \mathbb{B}_{E_\xi}(f_0,R)$. By {\bf Step~1}, there are $T_1=T_1(2R+\|f_0\|_{E_\xi})>0$ and $\kappa_0\ge 1$  such that 
\begin{equation}
	T_1<t^+(f) \;\;\text{ for any }\;\; f\in \bar{\mathbb{B}}_{E_\xi}(f_0,2R) \label{FE1}
\end{equation}
and
\begin{align}\label{FE2}
	\|\phi(t;f_1)-\phi(t;f_2)\|_{E_\xi}\leq \kappa_0 \|f_1-f_2\|_{E_\xi}\,,\quad t\in [ 0,T_1]\,, \quad (f_1, f_2)\in \big( \bar{\mathbb{B}}_{E_\xi}(f_0,2R) \big)^2\,.
\end{align}
Let $N_1\in\mathbb{N}\setminus\{0\}$ be such that $(N_1-1)T_1< t_*\le N_1 T_1$. We claim that there exists $k_0\ge 1$ such that
\begin{itemize}
	\item[($\alpha$)] $t_*<t^+(f_1)$ for each $f_1\in \mathbb{B}_{E_\xi}\big( f_0,R \kappa_0^{1-N_1} \big)$,
	\item[($\beta$)] $\|\phi(t;f_1)-\phi(t;f_0)\|_{E_\xi} \leq k_0 \|f_1-f_0\|_{E_\xi}$ for $0\le t\le t_*$ and $f_1\in \mathbb{B}_{E_\xi}\big(f_0,R \kappa_0^{1-N_1}\big)$.
\end{itemize}
Indeed, let $f_1\in \mathbb{B}_{E_\xi}\big(f_0,R \kappa_0^{1-N_1}\big) \subset \mathbb{B}_{E_\xi}\big(f_0,R\big)$. Either $t_*\le T_1$  (i.e., $N_1=1$) and the properties~($\alpha$) and~($\beta$) follow from \eqref{FE1} and \eqref{FE2}. Or $T_1<t_*$, from which we deduce that $N_1\ge 2$ and $\phi(T_1;f_0)\in \mathbb{B}_{E_\xi}(f_0,R)$. Consequently, by \eqref{FE2},
\begin{align*}
	\|\phi(T_1;f_1) - f_0\|_{E_\xi} & \le \|\phi(T_1;f_1) - \phi(T_1; f_0)\|_{E_\xi} + \|\phi(T_1;f_0) - f_0\|_{E_\xi} \\
	& < \kappa_0 \|f_1-f_0\|_{E_\xi} + R \le R \kappa_0^{2-N_1} + R \le 2 R\,,
\end{align*}  
so that $\phi(T_1;f_1)\in \mathbb{B}_{E_\xi}(f_0,2R)$. Recalling \eqref{FE1}, one has $T_1<t^+(\phi(T_1;f_i))$ for $i=0,1$, while uniqueness of solutions to~\eqref{CP} entails that $\phi(t;\phi(T_1;f_i))=\phi(t+T_1;f_i)$ for $0\le t\le T_1$ and $i=0,1$. An application of \eqref{FE2} then gives
\begin{equation*}
	\|\phi(t+T_1;f_1)-\phi(t+T_1;f_0)\|_{E_\xi} \leq \kappa_0 \|\phi(T_1;f_1) - \phi(T_1;f_0)\|_{E_\xi} \le \kappa_0^2 \|f_1-f_0\|_{E_\xi}
\end{equation*}
for $t\in [0,T_1]$. If  $N_1=2$, then the properties~($\alpha$) and~($\beta$) are proved with $k_0 = \kappa_0^2$. Otherwise, we proceed by induction to deduce ($\alpha$) and ($\beta$) after $N_1-1$ iterations. In particular, the property~($\alpha$) implies that $(0,t_*)\times \mathbb{B}_{E_\xi}\big( f_0,R \kappa_0^{1-N_1} \big)$ is a neighborhood of $(t_0,f_0)$ in 
\begin{equation*}
	\mathcal{D}:=\{(t,f)\,:\, 0\le t<t^+(f)\,,\,f\in E_\xi\}\,;
\end{equation*}
that is, $\mathcal{D}$ is open in $\mathbb{R}^+\times E_\xi$. Owing to this feature and~($\beta$), it is now immediate that the map $\phi\in C(\mathcal{D},E_\xi)$ defines a semiflow in $E_\xi$.

\bigskip

\noindent{\bf Step~4: Global Existence Criterion.} Let $f\in E_\xi$ and consider $0<t_0<T$ such that $t_0<t^+(f)$ and the corresponding solution $\phi=\phi(\cdot;f)$ to~\eqref{CP} satisfies
\begin{equation}\label{p2}
	\sup_{t_0<t<T\wedge t^+(f)}\| \phi(t)\|_{Y}<\infty\,.
\end{equation}
Assume for contradiction that $t^+(f)<T$. \Cref{L2} and \eqref{p2} then entail that 
\begin{equation*}
	\|\mathcal{K}(\phi(s))\|_{E_0}\le M_0\,,\qquad t_0<s< t^+(f)\,,
\end{equation*}
for some $M_0>0$. We readily obtain from \eqref{S} and \eqref{phi} that, for $t\in [t_0,t^+(f))$,
\begin{align*}
	\|\phi(t)\|_{E_\xi} & \le \|U(t-t_0) \phi(t_0)\|_{E_\xi} + \int_{t_0}^{t} \|U(t-s) \mathcal{K}(\phi)(s) \|_{E_\xi}\, \mathrm{d}s \\
	& \le c_0(T_0,\xi,\xi) \|\phi(t_0)\|_{E_\xi} + c_0(T_0,0,\xi) \int_{t_0}^{t} (t-s)^{-\xi} \|\mathcal{K}(\phi)(s)\|_{E_0}\, \mathrm{d}s \\
	& \le c_0(T_0,\xi,\xi) \|\phi(t_0)\|_{E_\xi} + \frac{c_0(T_0,0,\xi) M_0}{1-\xi} T_0^{1-\xi} =: R_0\,.
\end{align*}
{\bf Step~1} now implies that there is $T_1=T_1(R_0)>0$ such that $\phi$ exists at least on $[s,s+T_1]$ for every $t_0<s< t^+(f)$, contradicting the maximality of $t^+(f)$. Consequently, $t^+(f)\ge T$ as claimed and~\eqref{global} is proved.

Moreover, if \eqref{K1b} and \eqref{global1bb} are valid, then it follows from \Cref{L1} and \Cref{L2} that, for any $T>0$, there is $M(T)>0$ such that
\begin{equation*}
	\|\mathcal{K}(\phi(t))\|_{E_0}\le M(T) \|\phi(t)\|_{E_\theta}\,,\quad 0<t<T\wedge t^+(f)\,.
\end{equation*}
Combining this estimate with \eqref{S} and \eqref{phi}, we find, for $t\in (0,T\wedge t^+(f))$,
\begin{align*}
	\|\phi(t)\|_{E_\theta} & \le \|U(t) f\|_{E_\theta} + \int_0^t \|U(t-s) \mathcal{K}(\phi)(s)\|_{E_\theta}\, \mathrm{d}s \\
	& \le c_0(T,0,\theta) t^{-\theta} \|f\|_{E_0} + c_0(T,0,\theta) M(T) \int_0^t (t-s)^{-\theta} \|\phi(s)\|_{E_\theta}\, \mathrm{d}s\,.
\end{align*}
We then apply the singular Gronwall inequality, see \cite[II.Theorem~3.3.1]{Ama1995} or \cite[Lemma~7.1.1]{Hen1981} to conclude that $\phi$ satisfies \eqref{global} (with any $t_0\in \big(0,T\wedge t^+(f) \big)$). Consequently, $t^+(f)\ge T$ and, as $T>0$ is arbitrary, $t^+(f)=\infty$.

\bigskip

\noindent{\bf Step~5: Positivity.} We finally provide a proof of the non-negativity of solutions to~\eqref{CP} emanating from non-negative initial values. To this end, let us first note that \eqref{K1} and \Cref{L1} imply that, for $\psi\in E_\theta^+$,
\begin{align*}
	0\le \int_0^\infty k(x,y)\psi(y)\,\mathrm{d}y & \le \int_0^\infty k_* \frac{\ell(x)\ell(y)}{x+y+(x+y)^m} \psi(y)\,\mathrm{d}y \\
	& \le k_*\|\psi\|_{Y} \frac{\ell(x)}{x+x^m}  \\
	&\le c_3 \|\psi\|_{E_\theta} \frac{\ell(x)}{x+x^m} \,,\quad x>0\,.
\end{align*}
Assume first that the initial value $f$ belongs to $E_\theta^+$ and fix $T\in (0,t^+(f))$. By {\bf Step~1},
$$
\phi=\phi(\cdot;f)\in C([0,t^+(f)),E_\theta)
$$
and we may pick $R\ge \|\phi\|_{L_\infty((0,T),E_\theta)}$. Setting
\begin{equation*}
	V(x):= c_3 R\frac{\ell(x)}{x+x^m}\,,\quad x>0\,,
\end{equation*}
the above upper bound entails that, for $\psi\in \bar{\mathbb{B}}_{E_\theta}(0,R) \cap E_\theta^+$ and $x\in (0,\infty)$,
\begin{equation*}
	\big[ V\psi + \mathcal{K}(\psi) \big](x) \ge (V\psi)(x) - \psi(x) \int_0^\infty k(x,y) \psi(y)\, \mathrm{d}y \ge 0 \,.
\end{equation*}
Consider now the Cauchy problem
\begin{equation}\label{CP2}
	\frac{\mathrm{d}\psi}{\mathrm{d}t} = \mathbb{A}_V \psi + \mathcal{K}(\psi) + V\psi\,,\quad t>0\,,\qquad \psi(0)=f\,,
\end{equation}
with $\mathbb{A}_V=\mathbb{A}-V$. Since $\mathbb{A}_V$ generates a positive analytic semigroup on $E_0$ according to \Cref{L4} and 
\begin{equation*}
	\|V\psi\|_{E_0} \le c_3 R \|\psi\|_{Y}\,, \qquad \psi\in Y\,,
\end{equation*}
we may proceed as in {\bf Step~1} (with $\xi=\theta$) to obtain a unique mild solution $\psi\in C([0,T_1],E_\theta)$ to~\eqref{CP2} for some $T_1>0$, which satisfies additionally $\psi(t)\in E_\theta^+$ for $t\in [0,T_1]$. Clearly, $\psi=\phi$ on $[0,T\wedge T_1]$ by uniqueness of  mild solutions to~\eqref{CP}, and we iterate this argument to show that $\phi(t)\in E_\theta^+$ for $t\in [0,T]$. Since $T<t^+(f)$, the proof is complete for initial values in $E_\theta^+$. 
	
Finally, the just established non-negativity extends to arbitrary initial values in $E_\xi^+$ with $\xi\in (0,\theta)$ by a density argument, thanks to the continuous dependence of $\phi(\cdot;f)$ on the initial value $f$ established in {\bf Step~3}.
\end{proof}

%%%%%%%%%%%%%%%%
%%%%%%%%%%%%%%%%
\section{Global Existence}\label{sec3}
%%%%%%%%%%%%%%%%
%%%%%%%%%%%%%%%%

In this section, the rate coefficients $a$, $b$, and $k$ are assumed to satisfy \eqref{A.0}, \eqref{B.0}, \eqref{B.10}, \eqref{K1b}, and \eqref{K2} with parameters $0\le \theta_0 <\theta< 1$, $\theta_0\in [0,1/2]$, and $m>1$. We recall that \eqref{K1b} implies that $k$ satisfies also \eqref{K1} with $k_*=2K_*$. Throughout this section, $\kappa$ and $(\kappa_i)_{i\ge 1}$ denote positive constants depending only on $a$, $b$, $k$, $\theta_0$, and $\theta$. Dependence upon additional parameters will be indicated explicitly.

\bigskip

Let $r>1$. We first recall that, owing to \eqref{B.10}, 
\begin{equation}
	\delta_r := \inf_{y>0}\left\{ 1 - \frac{1}{y^r} \int_0^y x^r b(x,y)\, \mathrm{d}x \right\} \in (0,1)\,, \label{GBE0}
\end{equation}
with $\delta_r \in [\delta_2,1)$ for $r\ge 2$ and $\delta_r\in [1-(1-\delta_2)^{r-1},1)$ for $r\in (1,2)$, see \cite[Theorem~5.1.47]{BLL2020a}. We next define the weight 
\begin{equation*}
	w_r(x) := \left\{
	\begin{array}{ll}
		\displaystyle{\frac{r-1}{2} x^3}\,, & x\in [0,1]\,, \\[2ex]
		\displaystyle{x^r + \frac{r-3}{2} x}\,, & x\in (1,\infty)\,,
	\end{array}
	\right.
\end{equation*}
and notice that $w_r\in C^1([0,\infty))\cap C^2((0,\infty)\setminus\{1\})$ is nonnegative and increasing on $[0,\infty)$. Moreover, 
\begin{equation}
	\frac{x+x^r}{2} \le w_r(x) + r x\,, \quad w_r(x) \le x^r + rx\,, \qquad x\in (0,\infty)\,. \label{GK100}
\end{equation}

We now estimate how $w_r$ acts on the three mechanisms (fragmentation, diffusion, and coagulation) involved in \eqref{FD.0} and begin with the contribution of the fragmentation term. 

%%%%%%%%%%%%%%%%
\begin{lemma}\label{LGE1}
Consider $\psi\in E_1^+$ and $r\in (1,m]$. Then
\begin{equation*}
	\int_0^\infty w_r(x) \mathcal{F}(\psi)(x)\, \mathrm{d}x \le - \delta_r \int_1^\infty x^r a(x) \psi(x)\, \mathrm{d}x + \int_1^\infty x a(x) \psi(x)\, \mathrm{d}x\,.
\end{equation*}
\end{lemma}
%%%%%%%%%%%%%%%%

\begin{proof}
Owing to the definition of $w_r$,
\begin{align*}
	\int_0^\infty w_r(x) \mathcal{F}(\psi)(x)\, \mathrm{d}x & = - \frac{r-1}{2} \int_0^1 a(y) \psi(y) \left[ y^3 - \int_0^y x^3 b(x,y)\, \mathrm{d}x \right]\, \mathrm{d}y \\
	& \qquad - \int_1^\infty a(y) \left( y^r + \frac{r-3}{2} y \right) \psi(y)\, \mathrm{d}y \\
	& \qquad + \frac{r-1}{2} \int_1^\infty a(y) \psi(y) \int_0^1 x^3 b(x,y)\, \mathrm{d}x\mathrm{d}y \\
	& \qquad + \int_1^\infty a(y) \psi(y) \int_1^y \left( x^r + \frac{r-3}{2} x \right) b(x,y)\, \mathrm{d}x \mathrm{d}y\,.
\end{align*}
By \eqref{B.0}, 
\begin{equation*}
	\int_0^y x^3 b(x,y)\,\mathrm{d}x \le y^2 \int_0^y x b(x,y)\, \mathrm{d}x \le y^3\,, \qquad y>0\,,
\end{equation*}
which, together with the constraint $r>1$ and the non-negativity of $\psi$, implies that the first term on the right-hand side of the above identity is non-positive. Consequently, using once more \eqref{B.0}, as well as \eqref{GBE0},
\begin{align*}
	\int_0^\infty w_r(x) \mathcal{F}(\psi)(x)\, \mathrm{d}x & \le - \int_1^\infty a(y) y^r \psi(y)\, \mathrm{d}y - \frac{r-3}{2} \int_1^\infty a(y) \psi(y) \int_0^1 x b(x,y)\, \mathrm{d}x \mathrm{d}y \\
	& \qquad + \frac{r-1}{2} \int_1^\infty a(y) \psi(y) \int_0^1 x^3 b(x,y)\, \mathrm{d}x\mathrm{d}y \\
	& \qquad + \int_1^\infty a(y) \psi(y) \int_0^y x^r b(x,y)\, \mathrm{d}x \mathrm{d}y - \int_1^\infty a(y) \psi(y) \int_0^1 x^r b(x,y)\, \mathrm{d}x \mathrm{d}y\\
	& = - \int_1^\infty a(y) \psi(y) \left[ y^r - \int_0^y x^r b(x,y)\, \mathrm{d}x \right]\, \mathrm{d}y \\
	& \qquad + \int_1^\infty a(y) \psi(y) \int_0^1 \left( \frac{r-1}{2} x^3 - \frac{r-3}{2} x - x^r \right) b(x,y)\, \mathrm{d}x\mathrm{d}y \\
	& \le - \delta_r \int_1^\infty a(y) y^r \psi(y)\, \mathrm{d}y \\
	& \qquad + \int_1^\infty a(y) \psi(y) \int_0^1 \left( \frac{r-1}{2} x^3 - \frac{r-3}{2} x  \right) b(x,y)\, \mathrm{d}x\mathrm{d}y\,.
\end{align*}
Since $r>1$, another use of \eqref{B.0} gives
\begin{equation*}
	\int_0^1 \left( \frac{r-1}{2} x^3 - \frac{r-3}{2} x \right) b(x,y)\, \mathrm{d}x \le \int_0^1 x b(x,y)\, \mathrm{d}x \le y\,, \qquad y>1\,,
\end{equation*}
which completes the proof. 
\end{proof}
	
We now turn to the contribution of the diffusion term.

%%%%%%%%%%%%%%%%
\begin{lemma}\label{LGE2}
	Consider $\psi\in E_1^+$ and $r\in (1,m]$. Then
	\begin{equation*}
		- \int_0^\infty w_r(x) \partial_x^2\psi(x)\, \mathrm{d}x \ge -3r M_1(\psi) - r^2 \int_1^\infty x^{r-2} \psi(x)\, \mathrm{d}x\,.
	\end{equation*}
\end{lemma}
%%%%%%%%%%%%%%%%

\begin{proof}
First, recalling that
\begin{equation*}
	\partial_x\psi(x) = - \int_x^\infty \partial_y^2 \psi(y)\, \mathrm{d}y\,, \qquad \psi(x) = \int_x^\infty (y-x) \partial_y^2 \psi(y)\, \mathrm{d}y\,, \qquad x\in (0,\infty)\,,
\end{equation*}	
by \cite[Lemma~2.1]{LaWa2021}, the properties $\psi\in E_1$ and $r\in (1,m]$ ensure that
\begin{equation*}
	\lim_{x\to\infty} |w_r(x) \partial_x\psi(x)| \le (1+r) \lim_{x\to\infty} \int_x^\infty y^r |\partial_y^2 \psi(y)|\, \mathrm{d}y  = 0
\end{equation*}
and
\begin{equation*}
	\lim_{x\to\infty} |\partial_x w_r(x) \psi(x)| \le r \lim_{x\to\infty} \int_x^\infty y \, (1+y^{r-1}) |\partial_y^2 \psi(y)|\, \mathrm{d}y  = 0\,.
\end{equation*}
Therefore, since $w_r\in C^1([0,\infty))$ with $w_r(0)=\partial_x w_r(0)=0$, 
\begin{align*}
	- \int_0^\infty w_r(x) \partial_x^2\psi(x)\, \mathrm{d}x & = - \Big[ w_r(x) \partial_x\psi(x) \Big]_{x=0}^{x=\infty} + \Big[ \partial_x w_r(x) \psi(x) \Big]_{x=0}^{x=\infty}  - \int_0^\infty \psi(x) \partial_x^2 w_r(x)\, \mathrm{d}x \\
	& = - 3 (r-1) \int_0^1 x\psi(x)\, \mathrm{d}x - r(r-1) \int_1^\infty x^{r-2} \psi(x)\, \mathrm{d}x \\
	& \ge -3r M_1(\psi) - r^2 \int_1^\infty x^{r-2} \psi(x)\, \mathrm{d}x\,,
\end{align*}
as claimed. 
\end{proof}

We finally estimate the contribution of the nonlinear coagulation term which is, without much surprise, harder to handle.

%%%%%%%%%%%%%%%%
\begin{lemma}\label{LGE3}
	Consider $\psi\in E_1^+$ and $r\in (1,m]$. Then
	\begin{align*}
		\int_0^\infty w_r(x) \mathcal{K}(\psi)(x)\, \mathrm{d}x & \le \kappa_1(r) \left( M_1(\psi) + M_{1+(r-2)_+}(\psi) \right) \int_1^\infty x^r (1+a(x))^\theta \psi(x)\, \mathrm{d}x  \\
		& \qquad + \kappa_1(r) M_1(\psi)^2\,.
	\end{align*}
\end{lemma}
%%%%%%%%%%%%%%%%

\begin{proof}
A standard computation gives
\begin{equation*}
	\int_0^\infty w_r(x) \mathcal{K}(\psi)(x)\, \mathrm{d}x = \frac{1}{2} \int_0^\infty \int_0^\infty [w_r(x+y)-w_r(x)-w_r(y)] k(x,y) \psi(x) \psi(y)\, \mathrm{d}y\mathrm{d}x\,,
\end{equation*}
and we proceed differently for the contributions of large sizes, of small sizes, and the interactions between small and large sizes. We start with the former and infer from \eqref{K3} and \cite[Lemma~7.4.4]{BLL2020b} that
\begin{align*}
	I_{1}  &:= \frac{1}{2} \int_1^\infty \int_1^\infty [w_r(x+y)-w_r(x)-w_r(y)] k(x,y) \psi(x) \psi(y)\, \mathrm{d}y\mathrm{d}x \\
	& = \frac{1}{2} \int_1^\infty \int_1^\infty \big[ (x+y)^r - x^r - y^r \big] k(x,y) \psi(x) \psi(y)\, \mathrm{d}y\mathrm{d}x \\
	& \le \frac{\max\{r , 2^r-2\}}{2} \int_1^\infty \int_1^\infty \frac{x^r y + x y^r}{x+y} k(x,y) \psi(x) \psi(y)\, \mathrm{d}y\mathrm{d}x \\
	& \le \max\{r , 2^r-2\} k_0 \int_1^\infty \int_1^\infty \left[ \frac{x^{r+1} y^2}{(x+y)^2} + \frac{x^2 y^{r+1}}{(x+y)^2} \right] (1+a(x))^\theta \psi(x) \psi(y)\, \mathrm{d}y\mathrm{d}x \\
	& \le 2^r k_0 \int_1^\infty \int_1^\infty \left[ x^{r} y + x^2 y (x+y)^{r-2} \right] (1+a(x))^\theta \psi(x) \psi(y)\, \mathrm{d}y\mathrm{d}x \,.
\end{align*}
At this point, either $r\in (1,2]$ and 
\begin{subequations}\label{GK3}
\begin{align}
	I_1 & \le 2^r k_0\int_1^\infty \int_1^\infty \left[ x^{r} y + x^r y \right] (1+a(x))^\theta \psi(x) \psi(y)\, \mathrm{d}y\mathrm{d}x \nonumber\\
	& \le 2^{r+1} k_0 M_1(\psi) \int_1^\infty x^r (1+a(x))^\theta \psi(x)\, \mathrm{d}x\,. \label{GK3a}
\end{align}
Or $r>2$ and it follows from the basic inequality 
\begin{equation*}
	(x+y)^{r-2} \le 2^{r-2} (x^{r-2}+y^{r-2})\,, \qquad (x,y)\in (0,\infty)^2\,,
\end{equation*} 
that
\begin{align}
	I_1 & \le 2^r k_0 \int_1^\infty \int_1^\infty \left[ x^{r} y + 2^{r-2} \big( x^r y +x^2 y^{r-1} \big) \right] (1+a(x))^\theta \psi(x) \psi(y)\, \mathrm{d}y\mathrm{d}x \nonumber\\
	& \le 2^{2r-1} k_0 \int_1^\infty \int_1^\infty \left[ x^{r} y^{r-1} + x^r y^{r-1} \big) \right] (1+a(x))^\theta \psi(x) \psi(y)\, \mathrm{d}y\mathrm{d}x \nonumber\\
	& \le 4^r k_0 M_{r-1}(\psi) \int_1^\infty x^r (1+a(x))^\theta \psi(x)\, \mathrm{d}x \,. \label{GK3b}
\end{align}
\end{subequations}
We now study the contribution involving only small sizes. To this end, we observe that 
\begin{equation*}
	\partial_x^2 w_r(z) \le \kappa_2(r) z \;\;\text{ for a.a. }\;\; z\in (0,2)
\end{equation*} 
with $\kappa_2(r) := r(r+3) 2^{(r-3)_+}$, so that, for $(x,y)\in (0,1)^2$, 
\begin{equation*}
	w_r(x+y) - w_r(x) -w_r(y) = \int_0^x \int_0^y \partial_x^2 w_r(x_*+y_*)\, \mathrm{d}y_*\mathrm{d}x_* \le \kappa_2(r)  xy (x+y)\,.
\end{equation*}
Therefore, using also \eqref{K1},
\begin{align*}
	I_{2} & := \frac{1}{2} \int_0^1 \int_0^1 [w_r(x+y)-w_r(x)-w_r(y)] k(x,y) \psi(x) \psi(y)\, \mathrm{d}y\mathrm{d}x \\
	& \le \frac{ \kappa_2(r)}{2} \int_0^1 \int_0^1 xy (x+y) k(x,y) \psi(x) \psi(y)\, \mathrm{d}y\mathrm{d}x \\
	& \le  \kappa_2(r) k_* \int_0^1 \int_0^1 x^{2-2\theta_0} y^{2-2\theta_0} \frac{x+y}{x+y + (x+y)^m} \psi(x) \psi(y)\, \mathrm{d}y\mathrm{d}x \\
	& \le  \kappa_2(r) k_* \left( \int_0^1 x^{2-2\theta_0} \psi(x)\, \mathrm{d}x \right)^2\,.
\end{align*}
Since $\theta_0\in [0,1/2]$, we have $x^{2-2\theta_0} \le x$ for $x\in (0,1)$ and conclude that
\begin{equation}
	I_2 \le  \kappa_2(r) k_* M_1(\psi)^2\,. \label{GK5}
\end{equation}

We finally estimate the contribution to coagulation mixing small and large sizes. In that case, owing to the symmetry of $k$,
\begin{align*}
	I_{3} & := \frac{1}{2} \int_1^\infty \int_0^1 [w_r(x+y)-w_r(x)-w_r(y)] k(x,y) \psi(x) \psi(y)\, \mathrm{d}y\mathrm{d}x \\
	& \qquad + \frac{1}{2} \int_0^1\int_1^\infty [w_r(x+y)-w_r(x)-w_r(y)] k(x,y) \psi(x) \psi(y)\, \mathrm{d}y\mathrm{d}x \\
	& = \int_0^1 \int_1^\infty [w_r(x+y)-w_r(x)-w_r(y)] k(x,y) \psi(x) \psi(y)\, \mathrm{d}y\mathrm{d}x \,.
\end{align*}
Since 
\begin{align*}
	w_r(x+y) - w_r(x) - w_r(y) & = (x+y)^r - y^r + \frac{r-3}{2} x - \frac{r-1}{2} x^3 \\
	& \le (x+y)^r - x^r - y^r + x^r + \frac{r-3}{2} x \\
	& \le \max\{ r, 2^r-2\}\frac{x^r y + x y^r}{x+y} + r x
\end{align*}
for $(x,y)\in (0,1)\times (1,\infty)$ by \cite[Lemma~7.4.4]{BLL2020b}, we deduce from \eqref{K1} and the above inequality
\begin{align*}
	I_3 & \le 2^r k_* \int_0^1 \int_1^\infty \left( \frac{x^r y + x y^r}{x+y} +  x \right) \frac{x^{1-2\theta_0} (1+a(y))^\theta y^m}{x+y + (x+y)^m} \psi(x) \psi(y)\, \mathrm{d}y\mathrm{d}x \\
	& \le 2^r k_* \int_0^1 \int_1^\infty \left( \frac{2 x y^r}{y} +  x \right) x^{1-2\theta_0} (1+a(y))^\theta \psi(x) \psi(y)\, \mathrm{d}y\mathrm{d}x\,. 
\end{align*}
Recalling that $\theta_0\in [0,1/2]$, we further obtain
\begin{align}
	I_3 & \le 2^{r+2} k_* \int_0^1 \int_1^\infty x y^r (1+a(y))^\theta \psi(x) \psi(y)\, \mathrm{d}y\mathrm{d}x \nonumber \\
	& \le 2^{r+2} k_* M_1(\psi) \int_1^\infty y^r (1+a(y))^\theta \psi(y)\, \mathrm{d}y\,. \label{GK6}
\end{align}
Since
\begin{equation*}
	\int_0^\infty w_r(x) \mathcal{K}(\psi)(x)\, \mathrm{d}x  = I_1 + I_2 + I_3\,,
\end{equation*}
we collect \eqref{GK3}, \eqref{GK5}, and \eqref{GK6} to complete the proof of \Cref{LGE3}.
\end{proof}

Collecting the outcome of \Cref{LGE1}, \Cref{LGE2}, and \Cref{LGE3} leads us to the following estimate.

{\samepage
%%%%%%%%%%%%%%%%
\begin{lemma}\label{LGE4}
Consider $\psi\in E_1^+$ and $r\in (1,m]$. Then
\begin{align*}
	& \int_0^\infty w_r(x) \left[ \mathbb{A}\psi(x) + \mathcal{K}(\psi)(x) \right]\, \mathrm{d}x \\
	& \qquad\qquad \le \kappa_3(r) \left( 1 + M_1(\psi) + M_{1+(r-2)_+}(\psi) \right)^{1/(1-\theta)} \left( 1 + M_1(\psi) + \int_1^\infty w_r(x) \psi(x)\, \mathrm{d}x \right)\,,
\end{align*}
recalling that $\theta$ is defined in \eqref{K1L}.
\end{lemma}
%%%%%%%%%%%%%%%%
}

\begin{proof}
We infer from \Cref{LGE1}, \Cref{LGE2}, and \Cref{LGE3} that
\begin{equation}
\begin{split}
	& \int_0^\infty w_r(x) \left[ \mathbb{A}\psi(x) + \mathcal{K}(\psi)(x) \right]\, \mathrm{d}x \\
	& \qquad \le - \delta_r \int_1^\infty x^r a(x) \psi(x)\, \mathrm{d}x + \int_1^\infty x a(x) \psi(x)\, \mathrm{d}x \\
	& \qquad\qquad + 3r M_1(\psi) + r^2 \int_1^\infty x^{r-2} \psi(x)\, \mathrm{d}x\\	
	& \qquad\qquad + \kappa_1(r) \left( M_1(\psi) + M_{1+(r-2)_+}(\psi) \right) \int_1^\infty x^r (1+a(x))^\theta \psi(x)\, \mathrm{d}x + \kappa_1(r) M_1(\psi)^2\,.
\end{split} \label{GK110}
\end{equation}
On the one hand, it follows from \eqref{A.0} that, with $R_r^{r-1} := \max\{ 1, 2r/\delta_r\}$,
\begin{align}
	\int_1^\infty x a(x) \psi(x)\, \mathrm{d}x & \le \|a\|_{L_\infty(1,R_r)} \int_1^{R_r} x \psi(x)\, \mathrm{d}x + \frac{1}{R_r^{r-1}} \int_{R_r}^\infty x^r a(x) \psi(x)\, \mathrm{d}x \nonumber \\
	& \le \kappa(r) M_1(\psi) + \frac{\delta_r}{2} \int_{1}^\infty x^r a(x)  \psi(x)\, \mathrm{d}x \label{GK120}
\end{align}
and
\begin{equation}
	\int_1^\infty x^{r-2} \psi(x)\, \mathrm{d}x \le \int_1^\infty x^{1+(r-2)_+} \psi(x)\, \mathrm{d}x \le M_{1+(r-2)_+}(\psi)\,. \label{GK130}
\end{equation}
On the other hand, we infer from Young's inequality that
\begin{align*}
	& \kappa_1(r) \left( M_1(\psi) + M_{1+(r-2)_+}(\psi) \right) \int_1^\infty x^r (1+a(x))^\theta \psi(x)\, \mathrm{d}x \\
	& \qquad \le \kappa_1(r) \left( M_1(\psi) + M_{1+(r-2)_+}(\psi) \right) \int_1^\infty x^r \psi(x)\, \mathrm{d}x \\
	& \qquad\qquad + \kappa_1(r) \left( M_1(\psi) + M_{1+(r-2)_+}(\psi) \right) \int_{1}^\infty x^r a(x)^\theta \psi(x)\, \mathrm{d}x \\
	& \qquad \le \kappa_1(r) \left( M_1(\psi) + M_{1+(r-2)_+}(\psi) \right) \int_1^\infty x^r \psi(x)\, \mathrm{d}x \\
	& \qquad\qquad + \frac{\delta_r}{2} \int_{1}^\infty x^r a(x) \psi(x)\, \mathrm{d}x  \\	
	& \qquad\qquad + \left( \frac{2}{\delta_r} \right)^{\theta/(1-\theta)} \kappa_1(r)^{1/(1-\theta)} \left( M_1(\psi) + M_{1+(r-2)_+}(\psi) \right)^{1/(1-\theta)}\int_1^\infty x^r \psi(x)\, \mathrm{d}x \\
	& \qquad \le \frac{\delta_r}{2} \int_{1}^\infty x^r a(x) \psi(x)\, \mathrm{d}x + \kappa(r) \left( 1 + M_1(\psi) + M_{1+(r-2)_+}(\psi) \right)^{1/(1-\theta)} \int_1^\infty x^r \psi(x)\, \mathrm{d}x \,.
\end{align*}
Combining \eqref{GK110}, \eqref{GK120}, \eqref{GK130}, and the above inequality leads us to 
\begin{align*}
	& \int_0^\infty w_r(x) \left[ \mathbb{A}\psi(x) + \mathcal{K}(\psi)(x) \right]\, \mathrm{d}x \\
	& \qquad\qquad \le \kappa(r) \left( 1 + M_1(\psi) + M_{1+(r-2)_+}(\psi) \right)^{1/(1-\theta)} \left( 1 + \int_1^\infty x^r \psi(x)\, \mathrm{d}x \right)\,.
\end{align*}
We finally use \eqref{GK100} to complete the proof.
\end{proof}

To finish off the proof of \Cref{th2}, we shall use \Cref{LGE4} to derive time-dependent bounds on the solutions to \eqref{FD.0} in $E_0$. Bearing in mind that the first moment of solutions to \eqref{FD.0} does not vary with time according to \eqref{M.102}, it is then sufficient to obtain estimates in $L_1((0,\infty),w_m(x)\mathrm{d}x)$. Observing that the right-side of the inequality in \Cref{LGE4} depends linearly on $\psi$ when $r\in (1,2]$, a feature which is not available when $r>2$, an application of \Cref{LGE4} will only provide an estimate in $X_{\min\{m,2\}}$ in a first step. As this is not sufficient to conclude when $m>2$, we shall use an iterative procedure as in \cite{Ban2020} in that case.

\begin{proof}[Proof of \Cref{th2}: {$m\in (1,2]$}]
Let $\xi\in [0,1)$ with $2\theta<1+\xi$ and consider $f\in E_\xi^+$. According to \Cref{th1}, there is a unique classical solution $\phi=\phi(\cdot;f)$ to \eqref{FD.0} defined on $[0,t^+(f))$ and satisfying
\begin{equation}
	M_1(\phi(t)) = M_1(f) \le \|f\|_{E_0}\,, \qquad t\in [0,t^+(f))\,. \label{GK10} 
\end{equation}
Let $T>0$ and consider $t\in (0,T\wedge t^+(f))$. Then $\phi(t)\in E_1$ by \Cref{th1} and we infer from \eqref{FD.0}, \Cref{LGE4} with $r=m$, and \eqref{GK10} that
\begin{equation*}
	\frac{\mathrm{d}}{\mathrm{d}t} \int_0^\infty w_m(x) \phi(t,x)\, \mathrm{d}x  \le \kappa(\|f\|_{E_0}) \left( 1 + \int_0^\infty w_m(x) \phi(t,x)\, \mathrm{d}x \right)\,.
\end{equation*}
Integrating the above differential inequality and using \eqref{GK100} give
\begin{align*}
	\int_0^\infty w_m(x) \phi(t,x)\, \mathrm{d}x & \le \kappa(\|f\|_{E_0})  e^{\kappa(\|f\|_{E_0}) t} \left( 1 + \int_0^\infty w_m(x) f(x)\, \mathrm{d}x \right) \\
	& \le  \kappa(\|f\|_{E_0}) e^{\kappa(\|f\|_{E_0}) t} \left( 1 + m\|f \|_{E_0} \right)
\end{align*}
for $t\in (0,T\wedge t^+(f))$; that is, thanks to \eqref{GK100} and \eqref{GK10},
\begin{equation*}
	\|\phi(t)\|_{E_0} \le 2 \int_0^\infty \big( w_m(x) + mx \big) \phi(t,x)\, \mathrm{d}x \le \kappa(\|f\|_{E_0})  e^{\kappa(\|f\|_{E_0}) t}
\end{equation*} 
for $t\in (0,T\wedge t^+(f))$. We now infer from \Cref{th1}~\textbf{(d)} that $t^+(f)=\infty$.
\end{proof} 

To reach higher values of $m$, we proceed along the lines of \cite{Ban2020} and employ an iterative method. 

\begin{proof}[Proof of \Cref{th2}: $m>2$]
Keeping the notation used in the proof of  \Cref{th2} for $m\in (1,2]$, we introduce
\begin{equation*}
	\mu_r(T) := \sup_{t\in [0,T\wedge t^+(f))} \int_0^\infty (x+x^r) \phi(t,x)\, \mathrm{d}x \in (0,\infty]\,, \qquad T>0\,.
\end{equation*}
We first perform the same computation as in the previous proof with $m=2$ to obtain that
\begin{equation}
	\mu_2(T) < \infty\,, \qquad T>0\,. \label{GK20}
\end{equation}
We next take $r\in (2,m]$ and claim that, for each $T>0$,
\begin{equation}
	\mu_{r-1}(T) < \infty\ \Longrightarrow\ \mu_r(T) < \infty\,. \label{GK200}
\end{equation}
Thus suppose that $\mu_{r-1}(T) < \infty$ for some $T>0$. We then infer from \Cref{LGE4}, \eqref{GK10}, and \eqref{GK200} that, for $t\in (0,T\wedge t^+(f))$, 
\begin{align*}
		& \frac{\mathrm{d}}{\mathrm{d}t} \int_0^\infty w_r(x) \phi(t,x)\, \mathrm{d}x \\
		& \qquad \le \kappa_3(r) \left( 1 + M_1(\phi(t)) + M_{1+(r-2)_+}(\phi(t)) \right)^{1/(1-\theta)} \left( 1 + M_1(\phi(t)) + \int_1^\infty w_r(x) \phi(t,x)\, \mathrm{d}x \right) \\
		& \qquad \le \kappa_3(r) \left( 1 + M_1(f) + M_{r-1}(\phi(t)) \right)^{1/(1-\theta)} \left( 1 + M_1(f) + \int_1^\infty w_r(x) \phi(t,x)\, \mathrm{d}x \right) \\
		& \qquad \le \kappa(r,\|f\|_{E_0}) \left( 1 + \mu_{r-1}(T) \right)^{1/(1-\theta)}  \left( 1 + \int_0^\infty w_r(x) \phi(t,x)\, \mathrm{d}x \right)\,.
\end{align*}
After integrating with respect to time and using \eqref{GK100}, we end up with
\begin{equation*}
	\int_0^\infty (x + x^r) \phi(t,x)\, \mathrm{d}x \le \kappa(r,\|f\|_{E_0}) e^{\kappa(r,\|f\|_{E_0}) \left( 1 + \mu_{r-1}(T) \right)^{1/(1-\theta)} t}\,,  \qquad t\in (0,T\wedge t^+(f))\,. 
\end{equation*}
Consequently, recalling that $\mu_{r-1}(T) < \infty$,
\begin{equation*}
	\mu_r(T)<\infty\,,
\end{equation*}
and \eqref{GK200} is proved.

We now introduce $r_i := m - \lfloor m \rfloor + i$ for $1\le i \le \lfloor m \rfloor$ and note that $r_1\in [1,2]$. We then deduce from \eqref{GK10} and \eqref{GK20} that $\mu_{r_1}(T)<\infty$ for all $T>0$. We next apply \eqref{GK200} recursively to obtain that $\mu_{r_i}(T) < \infty$ for all $T>0$ and $1\le i \le \lfloor m\rfloor$. In particular, $\mu_{m}(T)  = \mu_{r_{\lfloor m\rfloor}}(T) < \infty$ for all $T>0$ and we have established \eqref{global1bb}, thereby completing the proof.
\end{proof} 

%%%%%%%%%%%%%%%%
%%%%%%%%%%%%%%%%
\section*{Acknowledgments}
%%%%%%%%%%%%%%%%
%%%%%%%%%%%%%%%%

This work was done while PhL enjoyed the kind hospitality and support of the Institut f\"ur Angewandte Mathematik, Leibniz Universit\"at Hannover.	

%%%%%%%%%%%%%%%%
%%%%%%%%%%%%%%%%
\bibliographystyle{siam}
\bibliography{CoagFragmDiff}
%%%%%%%%%%%%%%%%
%%%%%%%%%%%%%%%%
	
\end{document}